\documentclass[11pt, a4paper]{article}
\usepackage{amsmath,amssymb,amsthm}
\usepackage{mathtools}
\usepackage{amsfonts}
\usepackage{physics}
\usepackage{faktor}
\usepackage{bm}
\usepackage[T1]{fontenc}
\usepackage[utf8]{inputenc}
\usepackage[english]{babel}
\usepackage{kpfonts}
\usepackage{microtype}
\usepackage[autostyle]{csquotes}

\usepackage[style=alphabetic, url=false, doi=false,
date=year,backend=bibtex, maxbibnames=10, sorting=nyt]{biblatex}
\AtEveryBibitem{\clearfield{issn}}
\AtEveryCitekey{\clearfield{issn}}
\AtEveryBibitem{\clearfield{note}}
\AtEveryCitekey{\clearfield{note}}

\usepackage[toc]{appendix}

\usepackage{enumitem}
\usepackage{hyperref}
\usepackage{tabularx}

\usepackage[english, algosection, algoruled, noline]{algorithm2e}

\usepackage{cleveref}
\usepackage{xcolor}

\usepackage[margin=2.25cm]{geometry}

\theoremstyle{plain}
\newtheorem{theorem}{Theorem}[section]
\newtheorem{lemma}[theorem]{Lemma}
\newtheorem{prop}[theorem]{Proposition}
\newtheorem{corollary}[theorem]{Corollary}

\theoremstyle{definition}
\newtheorem{definition}[theorem]{Definition}

\theoremstyle{remark}
\newtheorem{remark}{Remark}

\newcommand{\N}{\mathbb N}

\newcommand{\R}{\mathbb R}

\renewcommand{\epsilon}{\varepsilon}

\newcommand{\cS}{\mathcal{S}}
\newcommand{\cO}{\mathcal{O}}
\newcommand{\cQ}{\mathcal{Q}}

\newcommand{\cL}{\mathcal{L}}

\renewcommand{\cS}{\mathcal{S}}

\newcommand*{\dist}[2]{\distance ({#1},{#2})}

\DeclareMathOperator{\pos}{Pos}
\DeclareMathOperator{\distance}{d}

\DeclareMathOperator{\Sing}{Sing}

\DeclareMathOperator{\jac}{Jac}
\DeclareMathOperator{\gram}{Gram}

\def \Gm {\widetilde{G}_-}

\setlist[enumerate,1]{label={(\roman*)},ref={\thetheorem (\roman*)}}

\def\NN{\mathbb{N}}
\def\RR{\mathbb{R}}

\def\RRg{\RR[\vb X]}

\newcommand{\cdummy}{\cdot}

\usepackage{amsfonts}
\usepackage[percent]{overpic}
\def\cst{\mathfrak{c}}
\def\eps{\epsilon}

\def\gam{\bm \gamma}
\def\lam{\bm \lambda}
\def\g{\vb g}
\def\h{\vb h}
\newcommand\inner[2]{\langle #1,#2·\rangle}

\def\DD{\dist{\vb g}{\Sing}}
\DeclareMathOperator{\diam}{diam}
\newcommand{\ca}[1]{\mathcal{#1}}
\usepackage{color,soul}
\begin{document}
\title{On Łojasiewicz Inequalities and the Effective Putinar's Positivstellensatz\thanks{This work has been partially supported by European Union’s Horizon 2020 research and innovation programme
under the Marie Skłodowska-Curie Actions, grant agreement 813211 (POEMA)}}
\author{Lorenzo Baldi\footnote{Inria d'Universit\'e C\^ote d'Azur, Sophia Antipolis and Sorbonne Universit\'e, CNRS, F-75005, Paris, France; partially funded by the Paris \^{I}le-de-France Region, under the grant agreement 2021-02--C21/1131}
, Bernard Mourrain\footnote{Inria d'Universit\'e C\^ote d'Azur, Sophia Antipolis},
Adam Parusiński\footnote{LJAD, UMR CNRS 7351. Universit\'e C\^ote d'Azur, Nice}
}
\date{\today}

\maketitle

\begin{abstract}
The representation of positive polynomials on a semi-algebraic set in terms of sums of squares is a central question in real algebraic geometry,  which the Positivstellensatz answers. In this paper, we study the effective Putinar's Positivestellensatz on a compact basic semi-algebraic set $S$ and provide a new proof and new improved bounds on the degree of the representation of positive polynomials. These new bounds involve a parameter $\eps$ measuring the non-vanishing of the positive function, the constant $\cst$ and exponent $L$ of a \L ojasiewicz inequality for the semi-algebraic distance function associated to the inequalities $\vb g = (g_1, \dots , g_r)$ defining $S$. They are polynomial in $\cst$ and $\eps^{-1}$ with an exponent depending only on $L$.
We analyse in details the \L ojasiewicz inequality when the defining inequalities $\vb g$ satisfy the Constraint Qualification Condition.  We show that, in this case, the \L ojasiewicz exponent $L$ is $1$ and we relate the \L ojasiewicz constant $\cst$ with the distance of $\vb g$ to the set of singular systems. 
\end{abstract}

\section{Introduction}\label{sec:1}
A fundamental difference between Algebraic Geometry and Real Algebraic Geometry is the use of the ordering of the real numbers. A central question in Real Algebraic Geometry is thus how to characterise real polynomials satisfying \emph{non-negativity} and \emph{positivity} conditions on a given domain, and not only those vanishing on it. This problem has attracted a lot of research in the last decades, also due to the connections with global optimization techniques. See e.g. \cite{lasserre_global_2001, marshall_positive_2008, laurent_sums_2009, lasserre_introduction_2015} or more recently \cite{powers_certificates_2021}. The purpose of this article is to present a quantitative version of Putinar's Positivstellensatz, a representation theorem for positive polynomials on a compact domain defined by polynomial inequalities.

The first example of globally non-negative polynomials $f \in \RRg = \RR[X_1, \dots, X_n]$ are the \emph{Sums of Squares} polynomials:
\[
    \Sigma^2 = \Sigma^2 [\vb X] \coloneqq \big\{ \, f \in \R[\vb{X}] \mid \exists r
\in \N, \ g_i \in \R[\vb{X}] \colon f = g_1^2 + \dots + g_r^2
\,\big\}.
\]
It is known since Hilbert \cite{hilbert_ueber_1888} that the convex cone of globally non-negative polynomials $\pos(\RR^n)$ contains properly the Sums of Squares (SoS) cone for $n \ge 2$, and the first explicit example of positive, non-SoS polynomial was given by Motzkin \cite{motzkin_t_s_arithmetic-geometric_1967}. The complete description of $\pos(\RR^n)$ in terms of SoS was proven by Artin \cite{artin_uber_1927}: $f \in \pos(\RR^n)$ if and only if $f$ can be written as a ratio of two SoS polynomials. This introduces a \emph{denominator} in the description of $f$.

In this paper we investigate the description of positive polynomials on \emph{basic closed semi-algebraic} sets:
\[
    S=\cS(\vb g) = \cS(g_1, \dots , g_r) \coloneqq \{ \, x \in \RR^n \mid g_1(x)\ge 0, \dots, g_r(x)\ge 0 \, \},
\]
in the particular case where $S$ is compact. Natural subcones of the cone $\pos(S)$ of non-negative polynomials on $S$ are the \emph{quadratic module}:
\[
    Q = \cQ(\vb g) \coloneqq \Sigma^2 + \Sigma^2 \cdot g_1+ \dots + \Sigma^2 \cdot g_r
\]
and the \emph{preordering} {\begin{align*}    
O &=\cO(\vb g) \coloneqq \cQ (\prod_{j \in J} g_j \colon  J \subset \{ 1,\dots , r \})\\
 &= \Sigma^2 + \Sigma^2 \cdot g_1+ \dots + \Sigma^2 \cdot g_r + \Sigma^2 \cdot g_1 g_2 + \dots + \Sigma^2 \cdot g_1 \dots g_r
\end{align*}}
While to characterize non-negative polynomials in terms of SoS and preorderings a denominator is necessary \cite{krivine_anneaux_1964, stengle_nullstellensatz_1974}, Schm\"udgen \cite{schmudgen_thek-moment_1991} showed that a denominator free representation exists for strictly positive polynomials on a basic compact semi-algebraic sets.
\begin{theorem}[{Schmüdgen's Positivstellensatz \cite{schmudgen_thek-moment_1991}}]
    \label{thm::schmudgen}
    Let $\cS(\vb g)$ be a compact basic semi-algebraic set. Then $f>0$ on $\cS(\vb g)$ implies $f\in \cO(\vb g)$.
\end{theorem}
This result greatly simplifies the representation. However, the representation still needs a number of SoS terms that is exponential in {$r$, the number} of defining inequalities of $S$, since the conclusion of the theorem is $f \in \cO(\vb g)$ and not $f \in \cQ(\vb g)$. The problem is solved when one introduces the \emph{Archimedean} property.
\begin{definition}
  \label{def::archimedean}
  Denote $\norm{\vb X}_2^2 = X_1^2 + \dots + X_n^2$. We say that a quadratic module $Q$ is \emph{Archimedean} if there exists $R \in \R$ such that $R^2-\norm{\vb X}_2^2 \in Q$.
\end{definition}
Notice that the Archimedean condition for $Q = \cQ(\vb g)$ implies the compactness of $S = \cS(\vb g)$. Moreover, as a corollary of \Cref{thm::schmudgen} we have that $\cO(\vb g)$ is Archimedean if $S$ is compact. This result is not true for quadratic modules: there are examples with $\cS(\vb
g)$ compact but $\cQ(\vb g)$ not Archimedean, see e.g.
\cite[ex.~6.3.1]{prestel_positive_2001}.

With the Archimedean condition, we can introduce the representation that we will study through the paper, based on the following theorem:
\begin{theorem}[{Putinar's Positivstellensatz \cite{putinar_positive_1993}}]
    \label{thm::putinar}
    Let $\cS(\vb g)$ be a basic semi-algebraic set. If $\cQ(\vb g)$ is Archimedean, then $f>0$ on $\cS(\vb g)$ implies $f\in \cQ(\vb g)$.
\end{theorem}
 
 The aim of the paper is to present a quantitative version of \Cref{thm::putinar}, giving an upper degree bound for the representation {$f = s_0 + s_1 g_1 + \dots + s_r g_r \in \cQ(\vb g)$} of a polynomial $f$ positive on $\cS(\vb g)$. This bound is presented in \Cref{thm::effective_putinar}. It involves $\epsilon = \epsilon(f)$, a measure for how $f$ is close to having a zero on $S$ {(see \Cref{parameters} (iv) for the definition)}, and a \L ojasiewicz exponent $L$ and coefficient $\cst$, that compare the behavior of $f$ and of the inequalities $g_1, \dots, g_r$ on a scaled simplex $D$ containing $S$.
{The \L ojasiewicz exponent and constant are defined in \Cref{def:loja}}
 
 {The dependence of degree bounds on continuous parameters, such as $\epsilon$ above, is typical of real algebraic geometry. In particular, lower degree bounds for the Positivstellens\"atze, showing the degrees of the SoS multipliers in the representation have to go to infinity as $\epsilon \to 0$, have been known since the work of Stengle \cite{Stengle1996}, for a special univariate example. Only recently, another quantitative lower degree bound in $\epsilon$ appeared in \cite{Baldi2024} for unit boxes. Let us also recall that for $\epsilon = \epsilon(f) = 0$, i.e. when the minimum of $f$ on $\cS(\vb g)$ is zero, there might be no representation of $f$ in the preordering or in the quadratic module (see e.g. \cite[Prop.~29]{Baldi2024}).}

 The problem of determining degree bounds for the Positvstellens\"ate is known as the \emph{Effective Putinar's Positivstellensatz} or \emph{Effective Schm\"udgen's Positivstellensatz}. While for lower degree bounds the only known results are above-mentioned \cite{Stengle1996,Baldi2024}, upper degree bounds have been extensively studied.
 
 For a special univariate example, the first upper degree bound can be found in \cite{Stengle1996}. 
 For general semialgebraic sets, upper degree bounds for the Effective Schm\"udgen's Positivstellensatz has been investigated for the first time by Schweighofer in \cite{schweighofer_complexity_2004}, while the Effective Putinar's Positivstellensatz by Nie and Schweighofer in \cite{nie_complexity_2007}. The bound obtained for Schm\"udgen's theorem were significantly better than those for Putinar's theorem: \cite{schweighofer_complexity_2004} has a polynomial dependence in $\epsilon$, while \cite{nie_complexity_2007} has an exponential one. It was an open question until recently if a polynomial dependence on $\epsilon$ was possible for Putinar's theorem: the first two authors gave a positive answer in \cite{baldi_moment_2021}.
 Upper bounds have also been studied for specific semi-algebraic sets, where special techniques can be applied to obtain better bounds: see for instance \cite{laurent_effective_2021,Baldi2024} for Schm\"udgen's and Putinar's theorems on the unit box, \cite{slot_sum--squares_2021} for Schm\"udgen's theorem on the unit ball and simplex and
\cite{fang_sum--squares_2020} for Putinar's theorem on the unit sphere.

{\L}ojasiewicz inequalities play a central role in the study of the Effective Positivstellens\"atze for general semialgebraic sets. Classical {\L}ojasiewicz inequalities are often stated as follows, see {\cite[cor.~2.6.7]{bochnak_real_1998}}.

\begin{theorem}
  \label{thm::lojasiewicz}
    Let $B$ be a closed bounded semi-algebraic set of $\RR^n$ and let $f, g$ be two continuous semi-algebraic functions from $B$ to $\RR$ such that $f^{-1}(0) \subset g^{-1}(0)$. Then there exists $c, L \in \RR_{>0}$ such that $\forall x \in B$:
\begin{align}\label{eq:loj_inequality}
      \abs{g(x)}^{L} \le c \abs{f(x)}.
\end{align}
 \end{theorem}
 {One can show that the smallest exponent $L$ for which the inequality  \eqref{eq:loj_inequality} holds always exists and is a strictly positive rational number (see \cite{lojasiewicz_1959}). It is called the \emph{{\L}ojasiewicz exponent}. Then, having $L$ fixed, the smallest $c>0$ such that this inequality holds, that also always exists, is called the \emph{{\L}ojasiewicz constant} (relative to $L$).}
 
 We apply the above {\L}ojasiewicz Inequality to three functions vanishing on 
 $S$, namely the function $F(x)$ defined in 
 \eqref{eq:defF}, the semi-algebraic distance to $S$, denoted $G(x)$ and defined in \eqref{eq:defG}, and the Euclidean distance function to $S$ denoted $E(x)$. 
 The Euclidean distance 
 to $S$, denoted $E(x)$, plays an auxiliary but fundamental role. 
 
In Theorem \ref{thm:main_theorem}, under the Constraint Qualification Conditions assumption, we give the {\L}ojasiewicz Inequality bound on $E(x)$ in terms of  $G(x)$.  It is known by \cite{baldi_moment_2021} that in this case $L=1$, and we give 
in Theorem \ref{thm:main_theorem} an explicit bound on the {\L}ojasiewicz constant. While the case of convex inequalities has been analyzed in the optimization community in \cite{lewis_error_1998},
the authors do not know any other reference where 
the {\L}ojasiewicz constant has been studied for general $\vb g$.  

{We also remark that the \L ojasiewicz inequality has been used to solve other problems in semialgebraic geometry constructively. For instance, in \cite{AverkovBröcker+2012+447+459} the \L ojasiewicz inequality is used to obtain bounds on the minimal number of polynomial inequalities defining a basic, closed semialgebraic set, in particular for the case of polyhedra.}

Finally, let us recall that the  {\L}ojasiewicz Inequality for the distance function to the zero set of a polynomial or a real analytic function is the original one and was introduced in the polynomial case by Hörmander \cite{hormander_1958} and in the analytic case by {\L}ojasiewicz \cite{lojasiewicz_1959}, in both cases to show the divisibility of Schwartz distributions by these functions.  Therefore such an inequality is sometimes called Hörmander-{\L}ojasiewicz Inequality.  

\subsection{Contributions and outline}
 In this paper, we develop a new analysis of the Effective Positivstellensatz, improving the existing upper bounds on the degree of representation of positive polynomials 
 and simplifying their descriptions.
 The approach improves the degree bounds obtained in  \cite{nie_complexity_2007, averkov_constructive_2013, kurdyka_convexifying_2015} from exponential bounds in $\epsilon^{-1}$ to a polynomial bound in $\epsilon^{-1}$, and the results in \cite{baldi_moment_2021}, by removing the dependency of the exponent of $\epsilon^{-1}$ on the dimension $n$. 

To obtain these new improved bounds, we analyse the \L ojasiewicz inequality connecting the semi-algebraic distance function $G$ associated to $\vb g$ and the distance function $F$ associated to $f$, that can be used directly in the proof of the Effective Positivstellensatz. Using a Markov inequality, we deduce a \L ojasiewicz inequality, which exponent is independent of $f$. The proof technique is similar to the one in \cite{baldi_moment_2021}. The main difference is the choice of a simpler semi-algebraic set containing $S$ that we reduce to. While in \cite{baldi_moment_2021} a unit box containing $S$ is used and a recent Effective Schm\"udgen's Positivstellensatz \cite{laurent_effective_2021} is applied, in the main  \Cref{thm::effective_putinar} we reduce to a simplex and apply an effective version's of Polya's theorem \cite{PowersnewboundPolya2001} (or the convergence property of the control polygon for the Bernstein basis). {In the study of the effective Positivstellens\"atze, another technique to reduce to the case of simplices has also been exploited in \cite{Schweighofer2002} and more recently in \cite{https://doi.org/10.48550/arxiv.2207.02748}.}

We analyze in detail the \L ojasiewicz inequality between $F$ and $G$ in the regular case, i.e. when the defining inequalities $\vb g$ satisfy the Constraint Qualification Condition. 
The main contribution in the regular case is \Cref{thm:main_theorem}, where the exponent is proven to be equal to one and an explicit bound for the constant in terms of geometric properties of the $\vb g$ is given. In \Cref{thm::CQC_distance_sing} we describe another interpretation of the constant as the distance from $\vb g$ to the set of singular systems, in the spirit of \cite{cucker_numerical_2009}.
 

In the remaining part of \Cref{sec:1}, we provide notation and preliminary material, and recall approximation properties needed in the proof of the Effective Positivstellensatz.
In \Cref{sec::loja}, we study \L ojasiewicz inequalities between different distance functions and analyse in detail {\L}ojasiewicz exponent and constant, when Constraint Qualification Conditions hold.
In \Cref{sec::effective_positiv} we prove the Effective Positivstellensatz and the new bound in \Cref{thm::effective_putinar}.
We conclude with additional remarks and perspectives in \Cref{sec:3.3}.




\subsection{Notation and conventions}
Let $\R[X_1,\ldots,X_n]=\R[\vb X]$ be the ring of polynomials in the variables $\vb X=(X_1, \ldots, X_n)$ with coefficients in $\R$. For $g_1,\dots,g_r \in \RRg$, let
$S = \cS(\vb g) \coloneqq \{\, x \in \RR^n \mid g_i(x) \ge 0,\ \forall i\in \{1,\dots , r\} \, \}$ be the basic semi-algebraic set defined by $g_1,\dots,g_r$.

Recall that a quadratic module $\cQ(\vb g)$ is called \emph{Archimedean}
if $R^2 - \norm{\vb X}_2^2 \in \cQ(\vb g)$ for some $R \in \RR$, see \Cref{def::archimedean}.
However, to simplify the proofs we assume that $R=1$.

\subsection*{Normalization assumption} 
\begin{equation}
  \label{assum::norm}
  \quad 1 - X_1^2 - \dots - X_n^2 \in \cQ(\vb g)s
\end{equation}
We can always be in this setting by a change of variables if we start with an Archimedean quadratic module.
Indeed, if $R^2 - \norm{\vb X}_2^2 \in \cQ(\vb g)$ then $1-\norm{\vb X}_2^2 \in  \cQ(\vb g(R \vb X))$
(i.e. the quadratic module generated by $g_i(R X_1, \dots, R X_n)$). Notice also that the normalization
assumption implies that $S$ is contained in the unit ball centered at the origin.

In the paper, we denote $$D \coloneqq \left\{ \, x\in \R^n\mid 1+ x_1\ge 0, \ldots, 1+x_n\ge 0, \sqrt{n} - x_1- \cdots- x_n\ge 0 \,\right\}$$ a simplex, containing the unit ball. Notice that $D \subset [-1,1+\sqrt{n}]^n$ and, by the normalization assumption, $S \subset D$.

For $f\in \R[\vb X]$ 
of degree $d = \deg(f)$ and $m \ge d$, we write $f= \sum_{\alpha\in \N^{n},|\alpha| \le  m} f_{m,\alpha} B_{m,\alpha}^D(\vb X)$ where $(B^D_{m,\alpha}(\vb X))_{\abs{\alpha}\le m}$ is the Bernstein basis in degree $m$ on $D$:
$$B^D_{m,\alpha}(\vb X)= {m \choose \alpha} (n+\sqrt{n})^{-d} (\sqrt{n}- X_1- \cdots- X_n)^{m-|\alpha|} (1+X_1)^{\alpha_1} \cdots (1+X_n)^{\alpha_n}$$
where ${m \choose \alpha}$ denotes the multinomial coefficient.

\subsection*{Norms} Hereafter we introduce the norms that will be used through the article.
\begin{itemize}
    \item For $f= \sum_{\alpha\in \N^{n},|\alpha| \le  m} f_{m,\alpha} B_{m,\alpha}^D(\vb X) \in \RRg$ and $m \ge \deg(f)$, we denote $\norm{f}_{B,m}$ the $L^{\infty}$ norm of $f$ with respect to the Bernstein basis:
    $$ \norm{f}_{B,m} = \max_{\abs{\alpha} \le m} \abs{f_{m,\alpha}}.$$
    When $m = \deg(f)$, we write $\norm{f}_B \coloneqq \norm{f}_{B,\deg(f)}$ to simplify the notation.
    \item For $f \in \RRg$, we denote $\norm{f}_\infty$ the infinity or supremum norm of $f$ on $D$:
    $$ \norm{f}_\infty = \max_{x \in D} \abs{f(x)}.$$
    \item For a vector $v = (v_1, \dots , v_N) \in \RR^N$, we denote $\norm{v}_2$ its Euclidean norm:
    $$ \norm{v}_2 = \sqrt{\sum_{i=1}^N v_1^2}.$$
    \item Moreover, if $M \in \RR^{N_1 \times N_2}$, we denote $\norm{M}_2$ the induced operator norm:
    $$ \norm{M}_2 = \sup_{v \neq 0} \frac{\norm{Mv}_2}{\norm{v}_2} = \sigma_{\max}(M),$$
    where $\sigma_{\max}(M)$ denotes the largest singular value of $M$.
\end{itemize} 

We recall some properties of the norms mentioned above, and in particular for the Bernstein norm that will be central in the article. For $f\in \RRg_m$ and $m' \ge m$, we have
$$
\max_{x\in D} |f(x)| = \norm{f}_\infty \le \|f\|_{B,m'} \le \|f\|_{B,m}
$$
These well-known inequalities are consequences of the property that the graph of $f$ is in the convex hull of its control points and that degree elevation representation is performed by barycentric combinations of the coefficients of $f$ (see e.g. \cite{FarinCurvessurfacesCAGD2001}).
We will also use the following multiplicative property of the Bernstein norm, which we briefly prove for the sake of completeness:
\begin{lemma}
For $f\in \RRg_m, g \in \RRg_{m'}$, we have
$$
\| f\, g\|_{B,m+m'} \le \|f\|_{B,m}\, \| g\|_{B,m'}
$$
\end{lemma}
\begin{proof}
For $f=\sum_{|\alpha|\le m} f_\alpha B_{m,\alpha}^D$, $g=\sum_{|\beta|\le m'} f_\beta B_{m',\beta}^D$,  we have
\begin{align*}
    \| f \, g\|_{B,m+m'} & = \| \sum_{|\gamma|\le m+m'} (\sum_{\alpha+\beta=\gamma} f_\alpha g_\beta \frac{{m \choose \alpha} {m'\choose \beta} }{{m+m' \choose \gamma}}) B^D_{\gamma}(x)\|_{B,m+m'}\\
    & = \max_{|\gamma|\le m+m'} |\sum_{\alpha+\beta=\gamma} f_\alpha g_\beta \frac{{m \choose \alpha} {m'\choose \beta} }{{m+m' \choose \gamma}}|\\
    & \le \max_{|\alpha|\le m} |f_\alpha| \, \max_{|\beta|\le m'} |g_\beta|  \max_{|\gamma|\le m+m'}\sum_{\alpha+\beta=\gamma} \frac{{m \choose \alpha} {m'\choose \beta} }{{m+m' \choose \gamma}} \le \| f\|_{B,m}\, \| g\|_{B,m'}
\end{align*}
\end{proof}

\subsection{Parameters} \label{parameters}

We summarize here the notation and symbols that will appear in the bound of the Effective Putinar's Positivstellensatz.
    \begin{enumerate}
    \item $\vb g = g_1, \dots , g_r$ denotes the $r$-tuple of real polynomials in $n$ variables defining the basic closed semialgebraic set $S = \cS(\vb g)$;
    \item $d(\vb g) \coloneqq \max_{i \in \{1, \dots, r\}} \deg(g_i)$;
    \item $f$ denotes a strictly positive polynomial on $S$ of degree $d = \deg(f)$ and $f^*=\inf\{ \, f(x)\mid x\in S \, \} > 0$ denotes its minimum on $S$;
    \item $\eps = \eps(f) \coloneqq  {\frac {f^*}  {\|f\|_{B}}}$ is a measure of how close $f$ is to vanish on $S$.
\end{enumerate}

In the article, by $O(\cdot)$, we mean a quantity such that  $O(\cdot) \le c_1 (\cdot)$ for some constant $c_1>0$ independent of $n$ and of the polynomials $\vb g, f$ involved in the problem.

\section{\L ojasiewicz inequalities for sum of squares representations}
\label{sec::loja}
{In this section we introduce several {\L}ojasiewicz inequalities between functions defined on $D$ and vanishing on $S= \cS(\vb g)$. In the following section, in order to analyze the representation of a positive polynomial $f$ on $S$, we use {\L}ojasiewicz inequalities to construct a polynomial $p$, a deformation of $f$, which is positive on $D $ with a minimum of the same order than $f^*=\inf_{x\in S} f(x)$. 
For this purpose, we need to compare on $D $ the behavior of the function $f$ with the behavior of the functions $g_1, \ldots, g_r$, and we introduce the following semi-algebraic functions.}
For $x\in D $, let 
\begin{align}\label{eq:defF}
F(x)&= - \min\left(\frac{f(x)-f^*}{\| f\|_B},0\right)\\ \label{eq:defG}
G(x)&= -\min\left(\frac{g_1(x)}{\|g_1\|_B}, \ldots, \frac{g_r(x)}{\|g_r\|_B},0\right).
\end{align}
The function $G$ can be seen as a \emph{semi-algebraic distance to} $S$, since $x\in S$ if and only if $G(x)=0$. 
{As $F(x)\ge 0$, $G(x)\ge 0$, $G^{-1}(0)=S$, and $F^{-1}(0)\supset S$ we deduce from  \Cref{thm::lojasiewicz} and the remark after it the existence of the following constants.}

\begin{definition}[{{\L ojasiewicz exponent and constant}}] \label{def:loja}
The smallest $L$ such that
\begin{equation}\label{eq:loja_F_G}
    \forall x \in D, \quad F(x)^L \leq \cst G(x)
\end{equation}
is called \emph{the {\L}ojasiewicz exponent}. {For $L$ satisfying \Cref{eq:loja_F_G} fixed,
we call the smallest constant $ \cst>0$ satisfying \Cref{eq:loja_F_G} \emph{the {\L}ojasiewicz constant} (relative to $L$).} 
\end{definition}

To analyse these exponent and constant, we first relate $F$ to the \emph{Euclidean distance} function \[ E \colon D \ni x \mapsto E(x) = \dist{x}{S}. \]
{This is another continuous semialgebraic function vanishing on $S$ and, therefore, $F$ and 
$E$ can be related by {\L}ojasiewicz inequality.  As we show below, we have  $\forall x\in D$, }
\begin{align}
\label{eq::F_E}
    F(x) & \le \frac{4d^2-2 d}{w(D)} E(x) \le 2 d^2 E(x),
\end{align}
with $w(D)=\sqrt{n}+1$ and $d=\deg(f)$.  {Let us first recall the following Markow inequality.}

\begin{theorem}[{\cite[th.~3]{kroo_bernstein_1999}}]
\label{thm::markov_inequality}
  Let $p \in \RRg_d$ be a polynomial of degree $\le d$. Then:
  \[
    \norm{\norm{\grad p(x)}_2}_\infty=\max_{x \in D}{\norm{\grad p(x)}_2} \le \frac {2 d (2d-1)} {w(D)} \norm{p}_{\infty}
  \]
  where $w(D)$, the width of $D$, is the minimal distance between a pair of distinct parallel supporting hyperplanes.
\end{theorem}

Now for $y\in D$ and $z\in S$ such that $E(y)=\dist{y}{S}=\|y-z\|_2$, we have 
\[
F(y) = F(y)-F(z) \le \cL_F \| y-z\|_2 = \cL_F E(y), 
\]
where $\cL_F$ is the Lipschitz constant of $F$ on $D$.  {Since $\cL_F= \| f\|_B \max_{x \in D}{\norm{\grad f(x)}_2} $, the inequality \eqref{eq::F_E} follows from the above Markov inequality theorem applied to $p=f$.}

As $E(x)=0$ implies $G(x)=0$,  these two functions are related as well by a {\L}ojasiewicz inequality:
\begin{equation}\label{eq:loja_E_G}
    \forall x \in D, \quad E(x)^{L_{E,G}} \leq \cst_{E,G} G(x)
\end{equation}
Therefore we can bound the {\L}ojasiewicz exponent and constant for $F$ and $G$, by analysing the {\L}ojasiewicz inequality between the \emph{Euclidean distance} function $E$
and the \emph{semi-algebraic distance} function $G$ in equation \eqref{eq:loja_E_G} and equation \eqref{eq::F_E}. More precisely, we have the following inequality: $L \le L_{E,G}$.

{In the next subsection, we analyze the {\L}ojasiewicz inequality \eqref{eq:loja_E_G} under a regularity assumption and show, that under this assumption,  $L_{E,G} = 1$.  We also compute 
 the constant $\cst_{E,G}$.}
Since $G$ and $S$ are invariant by scaling the functions $g_i$ by positive scalars, we will assume hereafter the following.
\subsection*{Scaling assumption:}
\begin{equation}
 \quad
  \label{assum::scaling}
  \norm{g_i}_B = 1 \text{ for all } i \in \{ \, 1,\dots , r \, \}
\end{equation}

\subsection{Minimizers of the distance function}
{In Definition \ref{def::CQC} below we introduce a regularity condition on $\vb g$ that implies  that $L_{E,G} = 1$, see Theorem \ref{thm:main_theorem}.} This is a standard condition in optimization (see \cite[sec.~3.3.1]{bertsekas_nonlinear_1999}), which implies the so-called Karush–Kuhn–Tucker (KKT) conditions \cite[prop.~3.3.1]{bertsekas_nonlinear_1999}.

\begin{definition}
\label{def::CQC}
Let $x \in \cS(\vb g)$. We define the \emph{active constraints} at $x$ are the constraints $g_{i_{1}}, \ldots, g_{i_{m}}$ such that $g_{i_{j}}(x)=0$.
We say that the \emph{Constraint Qualification Condition (CQC)} holds at $x$ if for all active constraints $g_{i_{1}}, \ldots, g_{i_{l}}$ at $x$, the gradients $\nabla g_{i_{1}}(x), \ldots,  \nabla g_{i_{m}}(x)$ are linearly independent.
\end{definition}

We start working locally. For $z \in S$ we denote \[ I = I(z) = \{ \, i \in \{1,\dots , r\} \mid g_i(z) =0\, \} \]
the indices corresponding to the \emph{active constraints} at $z$.
For $y \in D$ and $z \in S$ such that $E(y) = \norm{y-z}_2$
we denote:
\begin{itemize}
  \item $\vb g = \vb g (y) = (g_1(y), \dots, g_r(y))$;
  \item $\vb g_I = \vb g_I(y) = (g_i(y) \colon i \in I)$;
  \item $J = J(z) = \jac (\vb g_I)(z) = \bigl(\pdv{g_i}{x_j}\,(z)\bigr)_{i\in I, \ j \in \{1,\dots,r\}}$ the transposed Jacobian matrix of $\vb g$ at $z$,
  that is the matrix whose columns are the entries of the gradients $\grad g_i(z)$;
  \item $\vb N_I = \vb N_I(z) = \gram (\grad g_i(z) \colon i \in I) = J^t J$ the Gram matrix at $z$.
\end{itemize}

\begin{definition}\label{def::sing_val}
We denote by $\sigma_J (z) = \sigma_{\min}(J(z)) $ the smallest singular value $\sigma_{\min}(J(z))$ of $J(z)$.
\end{definition}
As $\vb N_I=J^t J$, notice that $\norm{\vb N_I^{-1}}_2= \sigma_{\min}(\vb N_I)^{-1}=\sigma_{\min}(J)^{-2}= \sigma_J(z)^{-2}$.

We show now how we can use $J = J(z)$
to describe the cone of points $y$ such that $E(y)=\dist{y}{S} = \norm{y-z}_2$.
\begin{lemma}
  \label{lem::difference_as_gradient}
    Let $y \in \RR^n \setminus \cS(\vb g)$, and let $z$ be a point in $S = \cS(\vb g)$
    minimizing the distance of $y$ to $S$, that is $E(y)=\dist{y}{S} = \norm{y-z}_2$. If $\{\,g_i \colon i \in I \,\}$ are the active constraints at $z$ and the CQC hold, then there exist $\lambda_i \in \RR_{\ge 0}$ such that:
    \[
      y - z = \sum_{i\in I} \lambda_i \grad (-g_i)(z) = - J \bm{\lambda}.
    \]
  \end{lemma}
  \begin{proof}
    Fix $y \in \RR^n$. Notice that $y - x = - \frac{\grad \norm{y-x}_2^2}{2}$, where the gradient is taken w.r.t. $x$. Moreover $z \in S$ such that $\dist{y}{S} = \norm{y-z}_2$ is a minimizer of the following Polynomial Optimization Problem:
    \[
      \min_x \frac{\norm{y-x}_2^2}{2} \colon g_i(x) \ge 0 \ \forall i\in \{1,\dots,r \}.
    \]
    Since the CQC holds at $z$, we deduce from \cite[prop.~3.3.1]{bertsekas_nonlinear_1999} that the KKT conditions hold. In particular:
    \[
      \frac{\grad \norm{y-z}_2^2}{2} = \sum_{i\in I} \lambda_i \grad g_i(z)
    \]
    For some $\lambda_i \in \RR_{\ge 0}$. Therefore $y-z = - \frac{\grad \dist{y}{z}^2}{2} = \sum_{i\in I} \lambda_i \grad (-g_i)(z)$.
  \end{proof}
Let  $\lam= \lam (y) := (\lambda_i (y); i\in I)$ be the column vector in \Cref{lem::difference_as_gradient}, so that 
$(y-z) = - J \lam $. Note that $\lam (y)$  depends linearly on  $y-z$ and is given by the formula 
$$
\lam (y) = - \vb N_I^{-1} J^t (y-z).
$$
Then, using Taylor's expansion at $z$ and \Cref{lem::difference_as_gradient}, we obtain:
\begin{equation}
\label{eq::g_Jacobian}
    \g_{I} = \g_{I} (y) = J^t (y-z)  + \h = - \vb N_I \lam + \h
\end{equation}
and the mean-value form for the remainder in Taylor's theorem gives:
\begin{align}\label{eq:hbound}
\norm{\vb h}_2\le \cst_2 \norm{y-z}_2^2, 
\end{align}
where $\cst_2 = \cst_2(\vb g) = \max_{x\in D}\{ \norm{\mathrm{Hess}( g_i)}_2, i=1, \dots,r\}$ denotes an upper bound for 
the second derivative of $\g_I$ on $D$.

We keep working locally at $z \in S$, and in particular considering only the active constraints at $z $,
whose indexes are denoted $I(z) \subset \{ \, 1, \dots , r \, \}$. Notice that, if $y \in D \setminus S$
is close enough to $z \in \partial S$, then $g_i(y) \le 0$ implies $g_i(z) = 0$: so only the active constraints
at $z$ and negative at $y$ determine the value of $G(y)$ in a neighborhood of $z$.
We introduce a notation to identify those indices:
  \begin{equation}\label{eq::I}
    I_- = I_-(y, z) = \{\, j \in I=I(z) \mid g_j (y) \le 0  \, \}.
  \end{equation}

Moreover we introduce the function $\Gm (y) = \bigl( \sum_{j \in I_-} g_j(y)^2 \bigr)^{\frac{1}{2}} $
as an intermediate step between $G$ and $E$. Indeed, it is easy to upper bound $\Gm (y)$ in terms of $G(y)$:
\begin{equation}\label{eq:G_between_norms}
    \Gm (y) = \bigl( \sum_{j \in I_-} g_j(y)^2 \bigr)^{\frac{1}{2}} \le \sqrt{\abs{I_-}} \max_{j\in I_-} \abs{g_j(y)} \le \sqrt{n} G(y).
    \end{equation}
    For the last inequality, we are using the fact that CQC at $z$ implies $\abs{I_-} \le \abs{I} \le n$. 
So we only need to find an upper bound for $E(y)$ in terms of $\Gm (y)$. In order to do that, let $\vb g_I (y)= \vb g_- (y)+ \vb g_+ (y)$, where:
\begin{itemize}
    \item $\vb g_- (y) = (\min\{0, \, g_i(y) \colon i \in I\})$ and
    \item $\vb g_+ (y) = (\max\{0, \, g_i(y) \colon i \in I\})$,
\end{itemize}
and notice that $\norm{\vb g_- (y)}_2 = \Gm (y)$.

We proceed similarly to analyze the linear part of $\vb g_I$. In the sequel we denote 
\begin{equation}
\label{eq::gamma}
    \gam = \gam (y)  = J^t (y-z) = - \vb N_I \lam (y) - \vb N_I \lam
\end{equation}
the linear part of $\g _I$ at $z$.

{To show Theorem \ref{thm:main_theorem} we first show the inequality 
\eqref{eq::basic_inequality} for the linear part $\gam (y)$, and then, in the following subsection, extend it to $\vb g_I$.}
In particular we want to relate the norm $\norm{y-z}_2 = \inner{y-z}{y-z}$,
computed with respect to the Euclidean scalar product,
with the norm of $\gam (y)$  w. r. t. another inner product.
Exploiting \eqref{eq::gamma}, one sees that 
\begin{equation}
    \label{eq::basic_inequality_prep}
    \inner{y-z}{y-z} = \inner{\bm \lambda}{\bm \lambda}_{\vb N_I} = \inner{\gam}{\gam}_{\vb N_I^{-1}}
\end{equation}
where $\inner{\cdummy}{\cdummy}_{\vb N_I}$ denotes the inner product induced by $\vb N_I$:
$\inner{\bm \lambda}{\bm \lambda}_{\vb N_I} = \bm \lambda^t N_I \bm \lambda$.
Notice that both $\vb N_I$ and $\vb N_I^{-1}$
define an inner product since they are positive definite.

As in the case of $\vb g_I$, let 
\begin{equation}\label{eq::I-}
  \tilde I_- = \tilde I_- (y, z) = \{ \, i\in I(z) \mid \gamma_i(y) \le 0 \, \}
\end{equation}
 and $\gam (y)= \gam_- (y)+ \gam_+ (y)$, where:
\begin{itemize}
    \item $\gam_- (y) = (\min\{0, \, \gam_i(y)\} \colon i \in I)$ and
    \item $\gam_+ (y) = (\max\{0, \, \gam_i(y) \}\colon i \in I)$.
\end{itemize}

\begin{lemma}\label{lem:gamma}
With the notation above, we have:
\begin{itemize}
    \item $\inner{\gam_-}{\gam}_{\vb N_I^{-1}} \ge 0$;
    \item $\inner{\gam_+}{\gam}_{\vb N_I^{-1}} \le 0$
    \item $\inner{\gam_+}{\gam_-}_{\vb N_I^{-1}} \le 0 $
\end{itemize}
\end{lemma}
\begin{proof}
For the first inequality notice that $\inner{\gam _-}{\gam}_{\vb N_I^{-1}} = - \gam_-^t \lam = - \sum_ {i\in \tilde I_-} \gamma_i \lambda_i \ge 0$
because all $\lambda_i$ are non-negative. A similar argument shows the second inequality. Finally $\inner{\gam_+}{\gam_-}_{\vb N_I^{-1}} = \inner{\gam_+}{\gam}_{\vb N_I^{-1}} - \inner{\gam_+}{\gam_+}_{\vb N_I^{-1}} \le 0$
as claimed. 
\end{proof}

  The following observation, crucial for the sequel, shows that we can bound
  $\norm{y-z}_2$ only in terms of 
  the negative $\gamma_i $. 

 \begin{prop}\label{lem:bound_Igam}
With the notation above, we have:
    \begin{equation}\label{eq::basic_inequality}
        \norm {y-z}_2\le \frac {1}{\sigma_J(z)} \big (\sum _{i\in \tilde I_-} \gamma^2_i (y) \big )^{1\over 2} = \frac {1}{\sigma_J(z)} \norm{\gam_{-}}_2
    \end{equation}
    where $\sigma_J (z)$ is the smallest singular value of $J$ (see \Cref{def::sing_val}).
  \end{prop}

\begin{proof}
Note that \Cref{lem:gamma} implies the proposition since it shows that
\[
    \inner{\gam}{\gam}_{\vb N_I^{-1}} = \inner{\gam_+}{\gam}_{\vb N_I^{-1}} + \inner{\gam_-}{\gam_+}_{\vb N_I^{-1}} + \inner{\gam_-}{\gam_-}_{\vb N_I^{-1}} \le \inner{\gam_-}{\gam_-}_{\vb N_I^{-1}}
\]
and this allows us to complete \eqref{eq::basic_inequality_prep} to get \eqref{eq::basic_inequality}:
\[
  \norm {y-z}_2 =\inner{y-z}{y-z} \le \inner{\gam}{\gam}_{\vb N_I^{-1}} \le \inner{\gam_-}{\gam_-}_{\vb N_I^{-1}} \le \frac {1}{\sigma_J(z)} \big (\sum _{i\in \tilde I_-} \gamma^2_i (y) \big )^{1\over 2} = \frac {1}{\sigma_J(z)} \norm{\gam_{-}}_2.
\]
\end{proof}

\subsection{{\L}ojasiewicz distance inequality}
\label{subsec::loja_distance}
{We now describe the {\L}ojasiewicz exponent $L_{E,G}$ and constant 
$\cst_{E,G}$ between $E$ and $G$ (see \eqref{eq:loja_E_G}) under the CQC assumption (\Cref{def::CQC}).  First note that, trivially, because $g_i$ are polynomials and $E$ is the Euclidean distance, $L_{E,G}\ge 1$.}

Let $\sigma_J =\inf_{z\in \partial S} \sigma_J(z) =\inf_{z\in \partial S} \sigma_{\min}(J(z))$. Notice that $\sigma_J > 0$ as $\partial S$ is compact and $\sigma_{\min}(J(z))$ is lower semicontinuous.
Let $I=I(z)$ and let $I_- = I_- (y) = \{i\in I \colon \g_i(y) \le 0\}$.  
Note that we do not have necessarily that $I_- = \tilde I_-$ (see \Cref{eq::I} and \Cref{eq::I-}):
the sign of $g_i(y)$ might be different from the sign of $\gamma_i(z)$.

{In Proposition \ref{lem:bound_Igam} we have obtained a bound in terms of the linear part  $\gam$ of $\vb g_I$.  
Now we are going to deduce from it an analogous bound in terms of $\vb g_I$.} To do this, we determine how close $\vb g_-$ and $\gam_-$ are.

\begin{lemma}\label{lem::small_diff}
With the notation above, we have:
\[
\big |\norm{\vb g_-}_2  - \norm{\gam_-}_2
 \big | \le \cst_2 \norm {y-z}_2^2.
\]    
\end{lemma}

\begin{proof}
Note that if $g_i (y)$ and $\gamma_i (y)$ are of different signs
then their absolute values are bounded by $\abs{g_i (y) - \gamma_i (y)}$.
Therefore, by standard triangle inequality,
\[
\big |\norm{\vb g_-}_2  - \norm{\gam_-}_2
 \big | = \big |\big (\sum _{i\in I_-} g^2_i (y) \big )^{1/2} - 
\big (\sum _{i\in \tilde I_-} \gamma^2_i (y) \big )^{1/2} \big | \le 
\big (\sum _{i\in I} (g_i (y) - \gamma_i (y))^2 \big )^{1/2}
= \norm{\vb h}_2\le \cst_2 \norm {y-z}_2^2,
\]   
where the latter inequality follows from \eqref{eq:hbound}.
\end{proof}
We first show the \L ojasiewicz inequality with $\textit{L}_{E,G} = 1$ locally at $z$.
\begin{prop}\label{prop:bound_near_S}
If $E(y) = \norm{y-z}_2 \le \frac {\sigma_J} {2 \cst_2}$ then 
\[
E(y)\le \frac {2\sqrt{n}} {\sigma_J} G(y).
\]
\end{prop}
\begin{proof}
Fix $y\not \in S$ such that $E(y) \le \frac {\sigma_J} {2 \cst_2}$ and $z\in \partial S$ such that 
$\norm {y-z}_2 = E(y)$.  
If $E(y) \le \frac {\sigma_J} {2 \cst_2}$ or, equivalently
$ \frac{\cst_2}{\sigma_J}  E^2(y) \le \frac 1 2 E(y)$,
 then by \Cref{lem:bound_Igam} and \Cref{lem::small_diff} we have
\begin{align*}
E(y) =   \norm{y-z}_2 & \le    \frac{1}{\sigma_J} \norm{\gam_{-}}_2
\le \frac{1}{\sigma_J} \norm{\vb g_{-}}_2 + \frac{1}{\sigma_J}  \cst_2 \norm {y-z}_2^2  \\ 
&  \le  \frac{1}{\sigma_J} \norm{\vb g_{-}}_2 +  \frac 1 2 E(y) .
\end{align*}
This implies the claimed inequality as $\norm{\vb g_{-}}_2 = \Gm (y) \le  \sqrt{n} \,  G(y)$
(since $|I_- (z,y)|\le |I(z)| \le n$ under CQC at $z$).
\end{proof}

We are finally able to prove that $L_{E,G} = 1$.
We denote $U  =\{ \, y \in D \mid E(y) < \frac {\sigma_J} {2 \cst_2} \, \}$ the open neighborhood of $S$
of points at distance $< \frac {\sigma_J} {2 \cst_2}$.

\begin{theorem} \label{thm:main_theorem}
Suppose that the CQC holds at every point of $\cS(\vb g)$.  Then, for all $y\in D$,
\[
E(y) \le \cst_{E,G}  G(y) , 
\]
with $\cst_{E,G} = \sup\{ \frac{E(y)}{G(y)}\mid y\in D\setminus S\} \le \max( \frac{2 \sqrt{n}}{\sigma_J}, \frac{\diam(D)}{G^*})$,  
where $\displaystyle G^*= \min_{y \in D \setminus U} G(z)>0$ and $\diam(D) = \max_{x,y \in D} \norm{x-y}_2$. 
\end{theorem}
\begin{proof}
If $E(y)\le \frac {\sigma_J} {2 \cst_2}$ then by Proposition \ref{prop:bound_near_S} we have
\[
E(y)\le  \frac{2 \sqrt{n}} {\sigma_J} G(y).
\]
Otherwise:
\[
E(y)= \|y-z\| \le \diam(D) \le \diam(D) {G(y) \over G^*},
\]
since $y,z\in D$ (notice that, as $G(x)>0$ on the compact set $D\setminus U$, we have $G^*>0$). 
\end{proof}

We want now to give another description of the constant $\cst_{E,G}$ in \Cref{thm:main_theorem} as the distance from \emph{singular systems}, following the approach of \cite{cucker_numerical_2009}.
In other words, we show how $\cst_{E,G}$ can be interpreted as the \emph{condition number} of $\vb g$.
See also \cite{burgisser_condition_2013} for more about condition numbers.

For $\vb d= (d_1,\ldots, d_r)$, let $\RRg_{\vb d} \coloneqq \RRg_{d_1}\times \cdots \times \RRg_{d_r}$
denote the systems of polynomials of bounded degree, which we equip with the Euclidean norm $\norm{\cdummy}_2$ with respect to the monomial basis in any component
(another choice could be the apolar or Bombieri-Weil norm $\norm{\cdummy}_{d_i}$ in degree $\le d_i$ in every component, see \cite{cucker_numerical_2009}).

We say that a system $\vb g$ is singular if there exists a point in $x \in \RR^n$ such that $x \in \cS(\vb g)$
and the active constraints have rank deficient Jacobian at $x$.
In other words, this is the set of systems $\vb g$ such that CQC
does not hold at some point of the semi-algebraic set $S$ defined by $\vb g$. Formally:
\begin{align}
\label{eq::sing_systems}
\begin{aligned}
    \Sing \coloneqq \big\{ \, \vb g \in \RRg_{\vb d} \mid \exists x \in \RR^n \colon \bigvee_{Z \subset \{ 1,\dots , r\}} & \big( g_j(x) = 0 \quad \forall j \in Z \\ &\land g_j(x) > 0 \quad \forall j \notin Z \\ &\land  \rank \jac (g_j (x) \colon j \in Z) < \min(n, \abs{Z})\big) \,\big\}
\end{aligned}
\end{align}

We want to relate the constant $\cst_D$ in \Cref{thm:main_theorem} with $\dist{\vb g}{\Sing}$,
the distance from $\vb g$ to the singular systems induced from the Euclidean norm.
Notice that $\Sing$ is a semi-algebraic set
(by Tarski--Seidenberg principle \cite[th.~2.2.1]{bochnak_real_1998}
or quantifier elimination\cite[prop.~5.2.2]{bochnak_real_1998}), and therefore $\dist{\cdot }{\Sing}$
is a well-defined continuous semi-algebraic function \cite[prop.~2.2.8]{bochnak_real_1998}.
\begin{lemma}
\label{lem::cst_J_dist}
    Under the normalization assumption \eqref{assum::norm} and with the previous notations,
    we have $\DD \le \sqrt{2} \sigma_{J}$.
\end{lemma}
\begin{proof}
Let $z \in \partial S$ be such that $\sigma_J = \sigma_{\min}(J(z))$.
Since the CQC hold at $z$, $\rank J(z)$ is maximal.
On the following, we assume that all the inequalities are active at $z$,
the general case being a trivial generalization.
By the Eckart-Young theorem, the distance of $J(z)$ from rank deficient matrices is equal to
$\sigma_{\min}(J(z))$:
there exists $P$ (of rank one) such that $J(z) - P$ has not maximal rank and
$\norm{P}_F = \norm{P}_2 = \sigma_{\min}(J(z))$.
Now consider a system $\vb l$ of affine equations vanishing at $z$ and such that
$\jac (\vb l)(z) = P$. Therefore, $\vb g - \vb l \in \Sing$ since
$\jac (\vb g - \vb l) (z) = J(z) - P$ is rank deficient and $(\vb g - \vb l)(z) = 0$.
Now, notice that:
\begin{equation*}
  \DD \le \norm{\vb g - (\vb g - \vb l)}_2 = \norm{\vb l}_2
\end{equation*}
Write $\vb l = l_1, \dots l_r$ and $l_i(x) = l_{i0} + \sum_{j=1}^n l_{ij} x_j$.
By hypothesis $l_i(z) = 0$ and $\norm{z}_2^2 \le 1$ (from the normalization assumption).
Therefore:
\[
  l_{i0}^2 = (\sum_{i=1}^n l_i x_i)^2 \le \norm{(l_{i1}, \dots , l_{in})}_2^2 \, \norm{z}_2^2\le  \sum_{j=1}^n l_{ij}^2
\]
Notice also that $\sigma_J^2 = \norm{P}_F^2 = \sum_{i=1}^r \sum_{j=1}^n l_{ij}^2$, and thus:
\[
  \DD^2 \le \norm{\vb l}_2^2 = \sum_{i=1}^r \sum_{j=1}^n l_{ij}^2 + \sum_{i=1}^r l_{i0}^2 \le 2 \sum_{i=1}^r \sum_{j=1}^n l_{ij}^2 = 2\sigma_J^2
\]
which concludes the proof.
\end{proof}
In order to measure the distance to $\Sing$, we introduce a global equivalent to $G^*$ in \cref{thm:main_theorem}.
We define then $\displaystyle \widetilde{G}^* \coloneqq \min_{y \in \RR^n \setminus U} G(z)>0$.

\begin{lemma}
\label{lem::G_interior}
    Let $U$ be as in \Cref{thm:main_theorem} and assume that $\widetilde{G}^* = G(y)$ is not attained on $ \partial U$. Then $\frac{1}{\widetilde{G}^*} \le \sqrt{r} \DD^{-1}$.
\end{lemma}
\begin{proof}
Without loss of generality assume that $g_1(y) = -\widetilde G^*$. Since $y\notin \partial U$ we have $\grad g_1 (y) = 0$. Then the system $(g_1 + \widetilde G^*,  \dots , g_r +\widetilde G^*) \in \Sing$ is a singular system, and $\norm{\vb g - (g_1 + \widetilde G^*, \dots , g_r + \widetilde G^*)}_2 = \sqrt{r}\, \widetilde G^*$. Therefore $\DD \le \sqrt{r}\, \widetilde G^*$, and finally $\frac{1}{\widetilde G^*} \le \frac {\sqrt{r}} {\DD}$.
\end{proof}
\begin{lemma}
\label{lem::g_boundary}
    Assume that $\widetilde{G}^* = G(y)$ is attained at $y \in \partial \{ \, y \in D \mid E(y) \le \frac {\sigma_J}{2 \cst_2} \, \}$. Then $\frac{1}{\widetilde{G}^*} \le \frac{4 \sqrt{n} \cst_2}{\sigma_J^2}$.
\end{lemma}
\begin{proof}
Since $E(y) = \frac {\sigma_J}{2 \cst_2}$, we can apply \Cref{prop:bound_near_S}: \[\frac {\sigma_J}{2 \cst_2} = E(y) \le 2 \sigma_J^{-1}\norm{\vb g_-}_2 \le 2 \sqrt{n} \, \sigma_J^{-1} G(y) = 2 \sqrt{n} \, \sigma_J^{-1} \widetilde G^*.\]
Therefore $\frac{1}{\widetilde G^*} \le 4 \sqrt{n} \cst_2 \sigma_J^{-2}$.
\end{proof}

We deduce from these two lemmas the following bound on \L ojasiewicz constant in terms of the distance from $\vb g$ to the singular systems $\Sing$:
\begin{theorem}
\label{thm::CQC_distance_sing}
Suppose that the CQC holds at every point of $\cS(\vb g)$.
Then, for all $y\in D$,
$$
E(y) \le  \max\big(\frac{\cst}{\DD}, \frac{8\diam(D) \sqrt{n}\, \cst_2}{\DD^2}\big)  G(y),
$$ 
where $\cst_2 = \cst_2(\vb g) = \max_{x\in D}\{ \norm{\mathrm{Hess}(g_i(x))}_2, i=1, \dots,r\} $ and $\cst_1 = \max({2\sqrt{2n}}, \diam (D)\sqrt{r})$.
\end{theorem}
\begin{proof}
  We estimate the constant
  $\cst_{E,G} = \sup\{ \frac{E(y)}{G(y)}\mid y\in D\setminus S\} \le \max( \frac{2 \sqrt{n}}{\sigma_J}, \frac{\diam(D)}{G^*})$
  in \Cref{thm:main_theorem} using the previous lemma. In particular, from \Cref{lem::cst_J_dist}
  we have $\frac{1}{\sigma_J} \le \frac{\sqrt{2}}{\DD}$, and
  using \Cref{lem::G_interior} and \Cref{lem::g_boundary} we obtain:
  \begin{align*}
    \frac{2 \sqrt{n}}{\sigma_J} & \le \frac{2 \sqrt{2n}}{\DD} \\
    \frac{\diam(D)}{G^*} &\le \frac{\diam(D)}{\widetilde G^*} \le \diam(D) \max(\frac{4\sqrt{n}\cst_2}{\sigma_J^2}, \frac{\sqrt{r}}{\DD}) \\ &\le \diam(D) \max(\frac{8\sqrt{n}\cst_2}{\DD^2}, \frac{\sqrt{r}}{\DD}) 
  \end{align*}
  Choosing $\cst_1 = \max({2\sqrt{2n}}, \diam (D)\sqrt{r})$ we then see that
  $\cst_{E,G} \le \max\big(\frac{\cst}{\DD}, \frac{8\diam(D) \sqrt{n}\, \cst_2}{\DD^2}\big)$,
  concluding the proof.
\end{proof}
\begin{remark}
  Under the CQC condition, we have analyzed in \Cref{thm:main_theorem} and \Cref{thm::CQC_distance_sing} the \L ojasiewicz constant, giving estimates for it,
  and moreover showing that the \L ojasiewicz exponent is equal to one. On the contrary when the problem is not regular the bounds on the exponent $L_{E,G}$ can be large. We have:
  \[
    L_{E,G} \le d(\vb g)(6d(\vb g) - 3)^{n+r}
  \]
  see \cite[sec.~3.1]{kurdyka_convexifying_2015}, \cite{kurdyka_metric_2016} and the errata \cite{kurdyka_correction_2019}.
  Recently, a new bound independent on the number of inequalities $r$ has been shown in \cite[th.~2]{basu_improved}:
  \[
    L_{E,G} \le d(\vb g)^{O(n^2)}.
  \]
  Finally, let us recall that the first quantitative estimation for the \L ojasiewicz inequality providing a bound with a single exponential in $n$ was given in \cite{solerno_effective}.
  \end{remark}

\begin{remark}
  {The function $G$ can be seen as semialgebraic distance to $S$, since $x\in S$ if and only if $G(x)=0$.
  Using the language of error bounds in optimization, the function $G$ can also be considered as a
  \emph{residual function}, see \cite{pang_error_1997}. Residual functions are used, in the analysis of iterative optimization algorithms, to bound the distance of an approximate solution from the true solution set}. {Using the language of error bounds and residual functions, a result analogous to \Cref{thm::CQC_distance_sing} has been proven in \cite[Prop.~7 and 8]{lewis_error_1998}, when $g_1, \dots , g_r$ are convex functions.}
\end{remark}
\begin{remark}
  {The CQC condition implies that the number of active constraints at every $z \in \cS(\vb g)$ is $\le n$.
  CQC also implies that for every point $y \in \RR^n$ with closest point  $z \in S$,  $y-z$ belongs to the convex cone generated by the gradients of the active constraints, see \Cref{lem::difference_as_gradient}.}
  
  {For convex sets $S$, the set of vectors $y-z$, for points $y$ whose closest point in $S$ is $z$, is called \emph{normal cone at $z$}. Abadie's Constraint Qualification (see e.g. \cite{pang_error_1997}) says that every vector in the normal cone is a conic combination of the gradients of the active constraints. This is the condition used in \cite[Prop.~7 and 8]{lewis_error_1998} to analyze the \L ojasiewicz exponent and constant for convex $g_1, \dots , g_r$.}
  
  {In this section, we could have similarly replaced the CQC condition with the (more general) assumption that for every point $y \in \RR^n$ with closest point  $z \in S$,  $y-z$ is a conic combination of the gradients of the active constraints. In other words, we could have assumed the conclusion of \Cref{lem::difference_as_gradient} instead of the CQC. Indeed, all the proofs of \Cref{sec::loja} can be adapted to this more general setting with minor changes.}
\end{remark}
\section{The Effective Positivstellensatz}
\label{sec::effective_positiv}
We analyze now how non-negative polynomials $\pos(S)$ 
can be approximated by polynomials that can be represented in terms of sums of squares. We quantify how the complexity of this representation, that is the degree of the terms, depends on the non-vanishing of the polynomial and {\L}ojasiewicz exponent and constant of $D$ and $G$.

For $l\in \N$, let $\Sigma^{2,l}\subset \R[\vb X]$ be the set of sums of squares of degree at most $l$, that is, the polynomials of the form $p= \sum_{i} p_i^2$ with $p_i\in \R[\vb X]$ of degree $\le {l\over 2}$. 
We define
$$
\cQ^l=\Sigma^{2,l} + (1-\sum_{i=1}^{n} X_i^2) \Sigma^{2,l-2} + g_1\Sigma^{2,l-d_1}+ \cdots + g_r \Sigma^{2,l-d_r},
$$
where $d_i=\deg(g_i)$ for $i=1,\ldots,r$.

{Recall from \eqref{eq:defG}, \eqref{eq:defF}, and using the notations from \Cref{parameters}, that 
\begin{align*}
    F(x) &= - \frac {1} {\|f\|_{B}} \min(f-f^*,0) \\
    G(x)&=-\min(\frac {g_1(x)} {\| g_1\|_{B}}, \ldots,\frac {g_r(x)} {\| g_r\|_{B}},0)= -\min( {g_1(x)}, \ldots, {g_r(x)},0)
\end{align*}
(by scaling $g_i$ we can assume that $\norm{g_i}_{B}=1$, see the scaling assumption \eqref{assum::scaling}).} We have $\forall x \in D$, $F(x)\ge 0$, $G(x)\ge 0$ and $\forall x \in S$, $F(x)=G(x)=0$. Moreover $G(x)=0$ implies that  $x\in S$ and $F(x)=0$. Also, $F(x)>0$ implies $G(x)>0$.
By the Łojasiewicz theorem, there exists $\cst_{F,G}>0, L_{F,G} \in \R$ such that $\forall x \in D$,
\begin{equation}\label{eq:loja ineq}
 {F(x)}^{L_{F,G}} \le \cst_{F,G} \, G(x).
\end{equation}

We are now ready to state the main result of the article.

\begin{theorem}[Effective Positivstellensatz]
\label{thm::effective_putinar}
 Let $f\in \R[\vb X]$ and $S=\{x\in D \mid g_1(x)\ge 0, \ldots, g_r(x)\ge 0\}$ with $S\subset B=\{x\in \R^n \mid 1-\sum_i x_i^2 \ge 0\}$. If $\forall x\in S$, $f(x)\ge f^* >0$, then $f\in \cQ^m$ for 
 $$
m = O( n^2\, r\, d(\vb g)^{6} \cst^7 \eps^{-(7 L +3) }),
 $$
 where $d(\vb g)= \max_i \deg(g_i)$, $\eps = {\frac {f^*}  {\|f\|_B}}$ and $\cst = \cst_{F,G}, L= L_{F,G}$ are respectively the {\L}ojasiewicz constant and exponent in  Inequality \eqref{eq:loja ineq}.

\end{theorem}

The proof follows the same lines as the proof of \cite[th.~1.7]{baldi_moment_2021}, but we work on the scaled simplex $D$ instead of the box $[-1,1]^n$ and we highlight the dependency of the bounds on {\L}ojasiewicz constant $\cst$, {\L}ojasiewicz exponent $L$ and on $\eps$. 

{
The main differences between \Cref{thm::effective_putinar} and \cite[th.~1.7]{baldi_moment_2021} are two:
\begin{itemize}
    \item we use a different \L ojasiewicz inequality, which leads to a smaller exponent $L$;
    \item we eliminate the explicit dependence on the number of variables $n$ in the exponent of $\epsilon$.
\end{itemize}
These improvements are achieved by introducing the semialgebraic function $F$, see \Cref{eq:defF}, and using the Bernstein norm instead of the max norm. For a more detailed comparison, we refer the reader to \Cref{sec:3.3}.
}
\subsection{Approximation of a plateau function}
The first ingredient is an approximation of a \emph{plateau} or Urysohn function by a sum of squares polynomial with a control of the error and of the degree of the polynomial. Recall that we are working under the scaling assumption \eqref{assum::scaling}: $\norm{g_i}_B=1$ for $i=1, \ldots, r$.

\begin{prop}\label{prop:def hi}For $i=1, \ldots, s$ and $\delta>0, \nu>0$, there exists $h_{i,\delta,\nu}\in \RRg$ such that 
\begin{itemize}
    \item For $g_i(x) \ge 0$,  $|h_{i,\delta,\nu}(x)| \le 2 \nu$. 
    \item For $ {g_i(x)} \le - \delta$, $|h_{i,\delta,\nu}(x)| \ge \frac 1 2$.
    \item $\| h_{i,\delta,\mu}\| \le 1$.
    \item $h_{i, \delta, \nu} \in \Sigma^{2,m}$ with $m =  O(n\, d(\vb g)^2 \delta^{-2} \nu^{- 1})$.
\end{itemize}
\end{prop}
\begin{proof}
To construct such a polynomial, we use the following plateau function.
For $\delta>0, \nu >0$, let $\varphi\in  C^0([-1, 1])$ be defined as: 
 $$
 \varphi = \left\{
 \begin{array}{lrcl}
 1 & -1 \le &x &\le -\delta\\
\sqrt  \nu + 3\, \frac {x^2} {\delta^2} (1- \sqrt \nu)
 +2\, \frac {x^3} {\delta^3} (1 - \sqrt \nu)
 & -\delta \le &x& \le 0\\
  \sqrt \nu & 0\le &x&  \le 1
 \end{array}
 \right.
 $$
 We verify that $\varphi$ is in $C^1([0,1])$, that $\varphi(0)= \sqrt \nu$, $\varphi(-\delta)= 1$ and that $\max_{x\in [-1,1]} |\varphi'(x)| \le \frac 2 \delta$.

 Let $\varphi_i = \varphi({g_i} )\in C^0(D)$. We have $\max_{x\in D}|\varphi_i|=1$. 
 
We are going to approximate $\varphi_i$ by a polynomial, using Bernstein operators defined in \eqref{def:bernstein op}. 
We deduce from Theorem \ref{thm:approx} that 
\begin{align}
|\ca B_m(\varphi_i;x)-\varphi_i(x)|& \le 2 \omega(\varphi_i; \frac{2 n} {\sqrt{m}}) \nonumber \\
&\le 2 \, \max_{x\in [-1, 1]} |\varphi'(x)|
\, \omega({g_i}; \frac{2 n} {\sqrt{m}})  \nonumber \\
& \le \frac 8 \delta  \max_{x\in D} {\| \nabla g_i(x)\|_2} \frac{n} {\sqrt{m}} \label{eq:bersntein approx bound}
\end{align}
Using Markov inequalities (Theorem \ref{thm::markov_inequality}),
we have that 
${\max_{x \in D}\| \nabla g_i(x)\|_2} \le  \frac{4\,  d(\vb g)^2}{\sqrt{n}+1} \, \|g_i\|_{\infty}\le \frac{4\,  d(\vb g)^2}{\sqrt{n}+1}$ where $d(\vb g)= \max_i \deg(g_i)$.  
Thus \eqref{eq:bersntein approx bound} implies that 
$$
|\ca B_m(\varphi_i;x)-\varphi_i(x)| \le 32 \,  n^{\frac 1 2}\,  d(\vb g)^2\, \delta^{-1} m^{- \frac 1 2}.
$$ 
Let us take $m'=O(n\, d(\vb g)^4 \delta^{-2} \nu^{-1})$, 
$$
s_i(x)= \ca B_{m'}^D(\varphi_i;x) 
=
\sum_{\alpha\in \N^n,|\alpha|\le m'} \varphi_i(\theta(\frac{\alpha}{m'})) \, B_{m',\alpha}^D(x)
$$
and $h_{i,\delta,\nu} = h_i = s_i^2$ so that 
for $x\in D$, 
$$
|s_i(x) - \varphi_i(x) |  \le \frac 1 4 \sqrt{\nu}
$$
Then we have
\begin{itemize}
     \item $g_i(x)\ge 0$  implies $s_i(x)\le \varphi_i(x)+ \frac 1 4 \sqrt \nu\le \frac 5 4 \sqrt \nu$ and $h_i(x)= s_i(x)^2 \le \frac {25} {16} \nu \le 2 \nu$.
    \item $ {g_i(x)}\le -\delta$ implies $s_i(x)\ge 1-\frac 1 4 \sqrt \nu$ and $h_i(x) \ge (1-\frac 1 4 \sqrt \nu)^2\ge \frac 1 2$ for $\nu$ small enough. 
    \item $\|h_i\| \le \|s_i\|^2 \le 1$.
    \item $h_i= s_i^2 \in  \Sigma^{2,m}$ with $m=2 m'= O(n\, d(\vb g)^2 \delta^{-2} \nu^{-1})$.
\end{itemize}
This concludes the proof of the proposition.
\end{proof}

\subsection{Exponents in the Effective Positivstellensatz}

We can now prove Theorem \ref{thm::effective_putinar}.

\begin{proof}[Proof of Theorem \ref{thm::effective_putinar}.]
Scaling $g_i$ by $\frac{1}{\norm{g_i}_{B}}$ does not change the definition of $S$ and the bound
in Theorem \ref{thm::effective_putinar}. Therefore we can assume hereafter that $\norm{g_{i}}_{B}=1$. 
Let 
\begin{equation}\label{eq:def_p}
p = f - \lambda \sum_{i=1}^{r} h_{i}\, {g_i},
\end{equation}
where $h_i=h_{i,\delta,\nu}$ is defined in \Cref{prop:def hi}.
We consider two cases:

1) $F(x)> {f^*\over 4 \|f\|_{B}}$. Then by {\L}ojasiewicz Inequality \eqref{eq:loja_F_G}, $G(x)> \delta :=  \cst^{-1} \epsilon^{L}$.
There exists $i\in [1,\ldots, r]$ such that $ g_i(x)\le - \delta$, say $i=1$. Then $h_1(x)\ge \frac 1 2$ and $h_i(x)\le 2\nu$ if $g_i(x)\ge 0$. We deduce that for $x\in D$, 
\begin{eqnarray*}
p(x) & \ge & f(x) + \lambda \frac 1 2 \delta - 2 \,\lambda\, {\nu} (r-1)\\
&\ge & f(x) + \lambda {\delta\over 4} + \lambda \left( {\delta \over 4} - {2\, \nu } (r-1) \right).
\end{eqnarray*}
Let $\lambda\gg 0$ such that
$f(x) + \lambda {\delta\over 4} \ge {1\over 4} f^*$
(i.e. $\lambda \ge {5 \over \delta} \|f\|_{B}$ since ${1\over 4}f^*-f(x)\le{5\over 4} \|f\|_{B}$) and 
let $\nu$ be small enough such that ${\delta \over 4} - {2 \nu} (r-1) \ge 0$ (i.e. ${\nu} \le {\delta\over  8 r} $).
Then $p(x)\ge {1\over 4} f^*$.

2) $F(x)\le \frac{f^*} {4 \| f\|_{B}}$. In this case, $f(x)\ge {3 \over 4} f^*$ and 
\begin{eqnarray*}
p(x) & \ge & {3 \over 4} f^* -  2 r\, \lambda \, {\nu}.  
\end{eqnarray*}
Then $p(x)\ge {1\over 4} f^*$ for $2\, r\, \lambda \, {\nu}  \le {1 \over 2} f^*$, i.e. $\nu \le  \, {f^* \over 4\, r\, \lambda }$.

We deduce that for $\delta= \cst^{-1} \epsilon^{L}$, $\lambda = 5 \delta^{-1} \|f\|_{B} = 5 \cst \epsilon^{-L} \|f\|_{B}$ and 
$$
\nu\le\min({\delta\over 8 r }, \, {f^* \over  4r \lambda })  = \frac {1} {20} r^{-1}\,\delta\, \epsilon = \frac {1} {20} \cst^{-1} \epsilon^{L+1},
$$ 
we have $p(x)\ge {1\over 4} f^*>0 $ for all $x\in D$.

Let $\eta = \max \{ \deg(f), \deg( h_i g_i), i = 1, \ldots, r\}$. Then 
$$
\deg(p) \le \eta = O(n\, d(\vb g)^3 \delta^{-2} \nu^{-1})
= O(n\, d(\vb g)^3 (\cst \epsilon^{-L})^2 (\cst \epsilon^{-(L+1)}))
= O(n\, d(\vb g)^3 \cst^3 \epsilon^{-(3L+1)}),
$$
and we have 
\begin{eqnarray*}
\|p\|_{B,\eta}=\| f- \lambda \sum_{i=1}^{r} h_{i}\, {g_i} \|_{B,\eta} &\le & \|f\|_{B,\eta} + \lambda \sum_{i=1}^{r}\|h_i\|_{B,\eta-\deg(g_i)}\, \| {g_i}\|_{B,\deg(g_i)} \le\|f\|_{B,\deg(f)} + \lambda r\\
& \le & (1 + 5 \cst \epsilon^{-L}  r) \, \|f\|_{B}
\le  (6\, r\,  \cst \epsilon^{-L})\, \|f\|_{B}
\end{eqnarray*}
for $\cst \epsilon^{-L}  \ge 1$.

Now, we use the property of convergence of the control polygon in the Bernstein basis on $D$ to the graph of the function.  By \Cref{cor:polya2} applied to the polynomial $p$ defined in \eqref{eq:def_p}, we get that for 
$$ 
m
= O(\deg(p)^2  \frac {\|p\|_{B,\eta}} {p^*}) 
= O(n^2 d(\vb g)^6\cst^6 \epsilon^{-(6L+2)}  r \cst \eps^{-L} \frac {\|f\|_{B}} {f^*})
= O(n^2 r\, d(\vb g)^6 \cst^7 \eps^{-(7L+3)})
$$
the Bernstein coefficients $p_{m,\alpha}$ in the Bernstein basis $(B^D_{m,\alpha})_{|\alpha|\le m}$ are non-negative. We deduce that $p$ belongs to the preordering generated by $1+ X_1, \ldots, 1+X_n, 1-n^{-\frac 1 2} (X_1+ \cdots+ X_n)$ in degree $m$ and conclude as in \cite[lem.~3.8]{baldi_moment_2021} that
$p\in \Sigma^{2,m+n} + (1-X_1^2- \cdots -X_n^2)\, \Sigma^{2,m+n-2}$. 
Therefore 
$f = p + \lambda \sum_{i=1}^{r} h_{i} {g_i}$ belongs to $\cQ^{m+n}$, which concludes the proof.
\end{proof}

\subsection{Some consequences, remarks and perspectives}\label{sec:3.3}

The exponent of $\epsilon^{-1}$ in the  bound of \Cref{thm::effective_putinar} depends on $f$. In order to avoid this dependency of the exponent on $f$, we can use a {\L}ojasiewicz inequality \eqref{eq:loja_E_G} between $E= \dist{\cdot}{S}$ and $G$ and a consequence of Markov inequality to get the following corollary:
\begin{corollary}\label{cor:effective_putinar_EG}

 Let $f\in \R[\vb X]$ with $d=\deg(f)$ and $S=\{x\in D \mid g_1(x)\ge 0, \ldots, g_r(x)\ge 0\}$ with $S\subset B=\{x\in \R^n \mid 1-\sum_i x_i^2 \ge 0\}$. If $\forall x\in S$, $f(x)\ge f^* >0$, then $f\in \cQ^m$ for 
 $$
m = O(  n^2\, r\, d(\vb g)^{6} 2^{\tilde L} d^{2 {\tilde L} } \tilde \cst^7 \eps^{-(7 \tilde L +3) })
 $$
 where $d(\vb g)= \max_i \deg(g_i)$, $\eps = {\frac {f^*}  {\|f\|}}$ and $\tilde \cst = \cst_{E,G}, \tilde L= L_{E,G}$ are respectively {\L}ojasiewicz constant and exponent in  Inequality \eqref{eq:loja_E_G}.
 
\end{corollary}
\begin{proof}
By inequalities \eqref{eq::F_E} and \eqref{eq:loja_E_G}, we have for $x\in D$,
$$
F(x)^{L_{E,G}} \le (2 d^2)^{L_{E,G}} E^{L_{E,G}} \le 2^{L_{E,G}} d^{2 {L_{E,G}} } \cst_{E,G} \, G(x).
$$
Applying \Cref{thm::effective_putinar} with $\cst = 2^{L_{E,G}} d^{2 {L_{E,G}} } \cst_{E,G}$ and $L= L_{E,G}$, we get the desired bound.
\end{proof}

We describe the main differences between \Cref{thm::effective_putinar} and the result in \cite{baldi_moment_2021}.

First, the bound in \Cref{thm::effective_putinar} uses \L ojasiewicz inequality \eqref{eq:loja_F_G}, while in \cite{baldi_moment_2021} 
the authors consider the \L ojasiewicz inequality \eqref{eq:loja_E_G}, as in \Cref{cor:effective_putinar_EG}. Not only it is more natural to work with  \L ojasiewicz inequality \eqref{eq:loja_F_G} instead of  \eqref{eq:loja_E_G}, but it also gives potentially significantly better bounds. For an illustration of this phenomenon, see e.g.  \cite[sec.~4]{de_klerk_error_2010}, where the authors discuss the gap between Lasserre's hierarchies based on the quadratic module and the preordering defining the unit hypercube.
Another advantage of \Cref{thm::effective_putinar} over \Cref{cor:effective_putinar_EG} is that in the case of an exact representation of $f-f^*$ in $Q(\vb g)$, we have $L=1$ (See. \Cref{prop:final_prop}).

On the other hand, \Cref{cor:effective_putinar_EG} allows to deduce a \emph{general} convergence rate for Lasserre's hierarchies, as done in \cite[sec.~4]{baldi_moment_2021}. The convergence rate that can be deduced from \Cref{cor:effective_putinar_EG} improves the one in \cite{baldi_moment_2021}, as there is no dependence on the number of variables $n$ in the exponent of $\epsilon$. As a corollary, we also obtain improved convergence rates for the  Haudorff distance of feasible pseudo moment sequences to moment sequences in the Lasserre's moment hierarchy, see \cite[sec.~5]{baldi_moment_2021}.

The second important difference between this article and \cite{baldi_moment_2021} is the norm used to define $\epsilon = \epsilon(f)$.
In this article, we use the max norm of the coefficients of $f$ in the Berstein basis on the scaled simplex $D$. This allows to use \cite{PowersnewboundPolya2001}, and leads to a bound with no $n$ in the exponent of $\epsilon$. On the other hand, in \cite{baldi_moment_2021} the norm used to define $\epsilon$ is the max norm on $[-1,1]^n$ and the approximation result in \cite{laurent_effective_2021} is exploited. This leads to a convergence rate with $n$ in the exponent of $\epsilon$. We can also rephrase \Cref{thm::effective_putinar} using the max norm on $D$ using the result in \cite{lyche_sup-norm_1997}, which can be stated, with our notation, as follows:
$$ \norm{f}_{B,d} \le K_d(\RR^n) \norm{f}_{\infty}$$
where $ K_d(\RR^n)$, given exactly in \cite[th.~4.2]{lyche_sup-norm_1997}, has asymptotic behaviour as in \cite[th.~5.1]{lyche_sup-norm_1997} when $d$ tends to infinity. Thus, if we use the norm $\norm{f}_{\infty}$ instead of $\norm{f}_{B,m}$ to define $\epsilon$, we need to multiply by the extra factor $K_m(\RR^n)^{-(7L+3)}$ in \Cref{thm::effective_putinar} and \Cref{cor:effective_putinar_EG}. 

When some regularity conditions hold, the bounds on the representation of positive polynomials can be simplified: if the CQC hold, we can apply the results of \Cref{subsec::loja_distance} and obtain the following corollary.
\begin{corollary}
\label{cor::effective_putinar_cqc}
    With the hypothesis of \Cref{thm::effective_putinar}, if the CQC hold for every $x \in S$ and $f>0$ on $S$ then $f \in \cQ^m$ for
     $$
m = O( n^2\, r\, d(\vb g)^6 \, \cst^7\, \eps^{-10}),
 $$
 where $\cst$ can be bounded using \eqref{eq::F_E} and \Cref{thm:main_theorem} or \Cref{thm::CQC_distance_sing}.
\end{corollary}

Notice that in \Cref{cor::effective_putinar_cqc} the exponent is \emph{independent} of $f$. In this case, the analysis performed to estimate the \L ojasiewicz exponent is then necessarily 
connecting the distance function $D$ and the euclidean distance $E$, rather than connecting directly 
$F$ and $G$.

The simplest case where we can apply \Cref{cor::effective_putinar_cqc} is when $S$ is the unit ball defined by the single polynomial $g = 1- \norm{\vb X}_2^2$.  This case can be analyzed, by specializing a general result, in \cite[cor.~1]{mai_complexity_2022}, where the authors prove a representation result for strictly positive polynomials (that includes furthermore a denominator) with degree of order $\epsilon^{-65}$.
In this case, \Cref{cor::effective_putinar_cqc} naturally gives a representation with order $\epsilon^{-10}$, improving \cite[cor.~1]{mai_complexity_2022}. In the case of the unit ball, to the best of our knowledge the best available result gives a bound of the order $\epsilon^{-1/2}$, and it is developed with a specific technique in \cite{slot_sum--squares_2021}.

As a perspective of this work, we would like to investigate the tightness of the bound. 
We can notice that if $f - f^*\in \cQ(\vb g)$  with $f^*\ge 0$ then $f \in \cQ_\ell(\vb g)$ for some $\ell \in \N$ and the bound on $\ell$ should not depend on $\epsilon$.  In this case, we see that $L \le 1$ as shown in the following proposition.
\begin{prop}\label{prop:final_prop}
    Assume that $f - f^* = s_0 + s_1 g_1 +\dots + s_r g_r \in \cQ(\vb g)$
    and let $F$, $G$ be as in \Cref{sec::loja}.
    Then:
    \[
        F(x) \le \cst G(x) \text{ for all } x \in D,
    \]
    where $\cst =\frac{1}{\norm{f}} \max \{ \, \sum_{i = 1}^r \norm{g_i} s_i(x) \colon x \in 
    D \, \}$.
\end{prop}
\begin{proof}
Let $f -f^* = s_0 + s_1 g_1 +\dots + s_r g_r \in \cQ(\vb g)$ with $s_i \in \Sigma^2$ and $x \in D$. There are two cases.

If $f(x) > f^*$ then $0 = F(x) \le \cst G(x)$ for any $\cst \in \RR_{>0}$.

If $f(x) \le f^*$ then $F(x) =\frac{f^* - f(x)}{\norm{f}} $. Therefore, if $\cst =\frac{1}{\norm{f}} \max \{ \, \sum_{i = 1}^r \norm{g_i} s_i(x) \mid x \in \Delta \, \}$ and $I_-(x) = \{ i \in \{1,\dots , r \colon g_i(x) \le 0\}$ we have:
\begin{align*}
    F(x) & = \frac{1}{\norm{f}}(f(x) - f^*)= - \frac{1}{\norm{f}} \big( s_0(x) + s_1(x) g_1(x) +\dots + s_r(x) g_r(x)\big) \\
    & \le \frac{1}{\norm{f}} \sum_{i \in I_-(x)} s_i(x) (-g_i(x)) = \frac{1}{\norm{f}} \sum_{i \in I_-(x)}\norm{g_i} s_i(x) \big(-\frac{g_i(x)}{\norm{g_i}} \big) \le \cst G(x).
\end{align*}
This shows that $F(x) \le \cst G(x)$ for all $x \in D$.
\end{proof}
This proposition suggests that the exponent of $\epsilon$ in a tight bound for the Effective Positivstellensatz  should vanish when $L=1$.

A particular case when $f-f^* \in \cQ_\ell(\vb g)$ is given by the so called \emph{Boundary Hessian Conditions} (BHC),
introduced by Marshall in \cite{marshall_representations_2006}. It would be interesting to see if, conversely, $L \le 1$ implies regularity conditions such as BHC and so that $f-f^* \in \cQ_\ell(\vb g)$. 

Another direction for future investigations is the analysis of worst case bounds in terms of the bit size and degree of the input polynomials with rational coefficients and to compare these bounds with those in \cite{LombardiElementaryRecursiveBound2020}. 

\appendix

\section{Approximation properties}

In this appendix, we recall and adapt to our context known approximation properties of continuous functions, focusing on our scaled simplex $D$. 

Let $\theta: x \in \Delta\mapsto (n+\sqrt{n})\, x - \vb 1 \in D$ be the affine map, which transforms the unit simplex $\Delta = \{ \, x \in \RR^n \mid x_i \ge 0 \text{ and } 1 - \sum_i x_i \ge 0 \, \}$ into $D$. For $m\in \NN$ and  $\psi \in C^0(D)$, consider the Bernstein operator 
\begin{equation}\label{def:bernstein op}
\ca B^D_m(\psi;x)= \sum_{\alpha\in \N^n,|\alpha|\le m} \psi\left(\theta(\frac{\alpha}{m})\right) B^{D}_{m,\alpha}(x)
\end{equation}

Notice that $\ca B_m^{D}$ is positive linear operator on $C^0(D)$, i.e. if $\forall x\in D, \psi(x)\ge 0$, then $\forall x\in D, \ca B^D_m(\psi;x)\ge 0$.
Moreover $\ca B^D_m$ reproduces constants and linear functions.

\begin{lemma}[{\cite[Lemma 4]{NewmanJacksonTheoremHigher1964}}]\label{lem:kernel}
For a positive linear operator $A: C^0(D) \rightarrow C^0(D)$, $\psi\in C^0(D)$ and $t>0$ and $x\in D$, we have
$$
| \psi(x) - A(\psi;x)| \le \omega(\psi;t) \big(1 + \frac 1 t \, A(\| \cdot -x\|_2^2;x)^{1\over 2}\big)
$$
where $\omega(\psi;\cdot)$ is the modulus of continuity of $\psi$.
\end{lemma}
Using this lemma for positive linear operators, we deduce the following approximation for the Bernstein operator: 
\begin{theorem}\label{thm:approx}
For $\psi \in C^0(D)$, 
$$
|\psi(x)- \ca B_{m}^D(\psi;x)| \le  2 \, \omega(\psi; \frac{2n} {\sqrt{m}}).
$$
\end{theorem}
\begin{proof}
First, using the property of the Bernstein operator on the unit simplex $\Delta$, we verify that 
$$
\|\theta^{-1}(x)\|_2^2 = \frac 1 {m(m-1)} \sum_{|\alpha|\le m} (\|\alpha\|_2^2 - m) B_{m,\alpha}^D(x) = \frac {m} {m-1} \ca B_m^{D}(\|\theta^{-1}(\cdot) \|_2^2;x) - \frac 1 {m-1}
$$
where $\theta^{-1}: x\in D \mapsto \frac{1} {n+\sqrt{n}} (x+\vb 1)$.
Therefore, we have
$$
\ca B_{m,\alpha}^D(\|\theta^{-1}(\cdot) \|_2^2, x) = \frac{1}{ (n+\sqrt{n})^2} \ca B_{m,\alpha}^D(\|\cdot + \vb 1 \|_2^2, x) =  \frac{m-1}{m} \|\theta^{-1}(x)\|_2^2 + \frac 1 {m}= \frac{m-1}{m} \|\frac{x+\vb 1}{ n+\sqrt{n}} \|_2^2 + \frac 1 {m}
$$
so that
$$
\ca B_{m}^D(\|\cdot + \vb 1 \|_2^2, x) =  \frac{m-1}{m} \|x+\vb 1\|_2^2 + \frac{(n+ \sqrt{n})^2} m.
$$
Since $\ca B_m^{D}$ reproduces affine functions, we have 
\begin{align*}
\ca B_m^{D}(\|\cdot -x\|_2^2;x) & = \ca B_m^{D}(\|(\cdot+ \vb 1) - (x+\vb 1)\|_2^2;x) \\
& = \frac{m-1}{m} \|x+ \vb 1\|_2^2 + \frac {(n+\sqrt{n})^2} {m} - 2 \|x+\vb 1\|_2^2 + \|x+\vb 1\|_2^2 \\ &= \frac 1 {m} ((n+\sqrt{n})^2 - \|x+ \vb 1\|^2 ) \le \frac {(n+\sqrt{n})^2} {m} \le 4 \frac {n^2} m
\end{align*}
for $x\in D$.
By \cref{lem:kernel}, we deduce that for $x\in D$,
$$
| \psi(x) - \ca  B_m^{D}(\psi;x)| \le 2\, \omega(\psi;\frac{2 \, n} {\sqrt{m}}) 
$$ 
choosing $t= \frac{2\, n} {\sqrt{m}}$.
\end{proof}

We recall now an effective version of Polya's theorem.
\begin{theorem}[{\cite[Th. 1]{PowersnewboundPolya2001}}]\label{thm:polya}
Let $p= \sum_{\abs{\beta} = d} p_{d,\beta} \binom{d}{\beta} X_0^{\beta_0} \dots X_n^{\beta_n}$ be an homogeneous polynomial of degree  $d = \deg(p)$. If $\forall x = (x_0, \dots, x_n) \in \RR^{n+1}$ such that $x_i \ge 0$ and $\sum_{i=0}^n x_i =1$, we have $p(x)\ge p^*>0$, then $(X_0+\dots+X_n)^m p$ has non negative coefficients in the monomial basis if
\[
    m \ge \frac{d(d-1)}{2} \frac{\max_\beta \abs{p_{d,\beta}}}{p^*} - d
\]
\end{theorem}

We can dehomogenize \Cref{thm:polya} setting $X_0 = 1 - X_1 - \dots - X_n$ and restate it using the Bernstein basis and norm, as follows.
\begin{corollary}\label{cor:polya1}
    Let $p \in \RRg$ be a polynomial of degree $d$. If $p \ge p^* > 0$ on $\Delta = \{ \, x \in \RR^n \mid x_i \ge 0, 1-\sum_i x_i \ge 0 \, \}$, then $p$ has non negative coefficients in the Bernstein basis $(B^\Delta_{m, \alpha})_{\abs{\alpha} \le m}$ if
    \[
        m \ge \frac{d(d-1)}{2} \frac{\norm{p}_{B}}{p^*} - d
    \]
\end{corollary}

\Cref{cor:polya1} can be seen as a result of the convergence of the control polygon to the graph of the polynomial $p$. Finally, we deduce from \Cref{cor:polya1} an analogous result for the case of the scaled simplex $D = \theta(\Delta) = \{x\in \R^n\mid 1+ x_1\ge 0, \ldots, 1+x_n\ge 0, \sqrt{n} - x_1- \cdots- x_n\ge 0\}$. In order do to that, notice that $B^D_{m, \alpha}(\theta(x)) = B^\Delta_{m, \alpha}(x)$ for all 
$x \in \Delta$ and $\alpha \in \N^n$ such that $\abs{\alpha} \le m$.
Furthermore, we state \Cref{cor:polya2} with a worst but simplified constant that will be more convenient in the next sections.

\begin{corollary}\label{cor:polya2}
Let $p= \sum_{\abs{\alpha} \le  m} p_{m,\alpha} B_{m,\alpha}^{D}(\vb X) \in \RRg$ with $m \ge d= \deg(p)$. If $\forall x\in D$ we have $p(x)\ge p^*>0$ and $m \ge d^2 \frac {\|p\|_{B}}{p^*}$, then $p_{m,\alpha}\ge 0$ for all $\alpha\in \N^{n}$ with $|\alpha|\le m$.
\end{corollary}

\printbibliography
\end{document}